\documentclass[12pt]{article}
\usepackage{amsmath,amsfonts,amsthm,amssymb,latexsym}
\usepackage[a4paper,margin=3cm]{geometry}
\usepackage{graphicx}
\usepackage{url}

\newcommand{\EM}{\ensuremath}

\newtheorem{thm}{Theorem}[section]%
\newtheorem{cor}[thm]{Corollary}%
\newtheorem{lem}[thm]{Lemma}%

\newtheorem{conj}[thm]{Conjecture}%

\newcommand{\dC}{\mathbb{C}}  
\newcommand{\dE}{\mathbb{E}}

 \newcommand{\dP}{\mathbb{P}} 
 \newcommand{\dR}{\mathbb{R}}

\newcommand{\rI}{\mathrm{I}}

\newcommand{\cA}{\mathcal{A}} 
\newcommand{\cC}{\mathcal{C}} \newcommand{\cD}{\mathcal{D}}
\newcommand{\cE}{\mathcal{E}} 
\newcommand{\cG}{\mathcal{G}} 
 
 \newcommand{\cL}{\mathcal{L}}
\newcommand{\cM}{\mathcal{M}} 
 \newcommand{\cP}{\mathcal{P}}
\newcommand{\cQ}{\mathcal{Q}} 
 
\newcommand{\cU}{\mathcal{U}} 
\newcommand{\cW}{\mathcal{W}}

\newcommand{\bA}{\mathbf{A}} 
 \newcommand{\bD}{\mathbf{D}}
\newcommand{\bE}{\mathbf{E}} 
 
\newcommand{\bI}{\mathbf{I}} 
 
\newcommand{\bM}{\mathbf{M}} 
 \newcommand{\bP}{\mathbf{P}}
\newcommand{\bQ}{\mathbf{Q}} 
\newcommand{\bS}{\mathbf{S}} \newcommand{\bT}{\mathbf{T}}
\newcommand{\bU}{\mathbf{U}} 
 \newcommand{\bX}{\mathbf{X}}

\newcommand{\ABS}[1]{{{\left| #1 \right|}}} 
\newcommand{\BRA}[1]{{{\left\{#1\right\}}}} 
\newcommand{\NRM}[1]{{{\left\| #1\right\|}}} 
\newcommand{\PAR}[1]{{{\left(#1\right)}}} 

\newcommand{\OL}[1]{\overline{#1}}

\newcommand{\SSK}[1]{\substack{#1}}
\renewcommand{\leq}{\leqslant}
\renewcommand{\geq}{\geqslant}

\newcommand{\ds}[1]{\EM{\displaystyle{#1}}}

\title{The Dirichlet Markov Ensemble} %
\author{Djalil~\textsc{Chafa\"\i}} %
\date{\small Preprint -- September 2007. Revised -- August 2009, October 2009.}

\begin{document}

\maketitle

\begin{abstract}
  We equip the polytope of $n\times n$ Markov matrices with the normalized
  trace of the Lebesgue measure of $\mathbb{R}^{n^2}$. This probability space
  provides random Markov matrices, with i.i.d.\ rows following the Dirichlet
  distribution of mean $(1/n,\ldots,1/n)$. We show that if $\bM$ is such a
  random matrix, then the empirical distribution built from the singular
  values of$\sqrt{n}\,\bM$ tends as $n\to\infty$ to a Wigner quarter--circle
  distribution. Some computer simulations reveal striking asymptotic spectral
  properties of such random matrices, still waiting for a rigorous
  mathematical analysis. In particular, we believe that with probability one,
  the empirical distribution of the complex spectrum of $\sqrt{n}\,\bM$ tends
  as $n\to\infty$ to the uniform distribution on the unit disc of the complex
  plane, and that moreover, the spectral gap of $\bM$ is of order
  $1-1/\sqrt{n}$ when $n$ is large.
\end{abstract}

{\footnotesize \textbf{AMS 2000 Mathematical Subject Classification:} 15A52;
  15A51; 15A42; 60F15; 62H99.}


\section{Introduction}
\label{se:intro}

Markov chains constitute an essential tool for the modelling of stochastic
phenomena in Biology, Computer Science, Engineering, and Physics. It is
nowadays well known that the trend to the equilibrium of ergodic Markov chains
is related to the spectral decomposition of their Markov transition matrix,
see for instance \cite{MR2209438,MR1490046,MR2375599}. The corresponding
literature is very rich, and many statisticians including for instance the
famous Persi Diaconis contributed to this subject, by providing quantitative
bounds for various concrete specific Markov chains. But how a Markov chain
behaves when its Markov transition matrix is taken arbitrarily in the set of
Markov matrices? The present paper aims to provide some partial answers to
this natural concrete question. From the statistical point of view, one can
think about considering random Markov matrices following the ``uniform law''
over the set of Markov matrices, which corresponds to a maximum entropy
distribution or Bayesian prior, see for example \cite{MR2278358}. Recall that
a $n\times n$ square real matrix $\bM$ is Markov if and only if its entries
are non--negative and each row sums up to $1$, i.e.\ if and only if each row
of $\bM$ belongs to the simplex
\begin{equation}\label{eq:simplex}
  \Lambda_n=\{(x_1,\ldots,x_n)\in[0,1]^n\text{ such that }x_1+\cdots+x_n=1\}
\end{equation}
which is the portion of the unit $\NRM{\cdot}_1$-sphere of $\dR^n$ with
non--negative coordinates. The spectrum of a Markov matrix lies in the unit
disc $\{z\in\dC;\ABS{z}\leq1\}$, contains $1$, and is symmetric with respect
to the real axis in the complex plane.

\subsubsection*{Uniform distribution on Markov matrices}

Let $\cM_n$ be the set of $n\times n$ Markov matrices. We need to give a
precise meaning to the notion of uniform distribution on $\cM_n$. This set is
a convex compact polytope with $n(n-1)$ degrees of freedom if $n>1$. It has
zero Lebesgue measure in $\dR^{n^2}$.

Since $\cM_n$ is a polytope of $\dR^{n^2}$ (i.e.\ intersection of half
spaces), the trace of the Lebesgue measure on it makes sense and coincides
with a cone measure\footnote{Actually, one can define the trace of the
  Lebesgue measure and then the uniform distribution on many compact subsets
  of the Euclidean space, by using the notion of Hausdorff measure
  \cite{MR2118797}. See also \cite{MR2481068} for an approximate simulation
  method based on billiards and random reflections.}, despite its zero
Lebesgue measure in $\dR^{n^2}$. Since $\cM_n$ is additionally compact, the
trace of the Lebesgue measure can be normalized into a probability
distribution. We thus define \emph{the uniform distribution} $\cU(\cM_n)$ on
$\cM_n$ as the normalized trace of the Lebesgue measure of $\dR^{n^2}$. The
following theorem relates $\cU(\cM_n)$ to the Dirichlet distribution.

\begin{thm}[Dirichlet Markov Ensemble]\label{th:DME}
  We have $\bM\sim\cU(\cM_n)$ if and only if the rows of $\bM$ are i.i.d.\ and
  follow the Dirichlet law of mean $(\frac{1}{n},\ldots,\frac{1}{n})$. The
  probability distribution $\cU(\cM_n)$ is invariant by permutations of rows
  and columns.
\end{thm}

\begin{cor}\label{co:DME}
  If $\bM\sim\cU(\cM_n)$ then for every $1\leq i,j\leq n$, $\bM_{i,j}\sim
  \mathrm{Beta}(1,n-1)$ and  for every $1\leq i,i',j,j'\leq n$,
  $$
  \mathrm{Cov}(\bM_{i,j},\bM_{i',j'}) =\begin{cases}
    0 & \text{if $i\neq i'$} \\
    \frac{n-1}{n^2(n+1)} & \text{if $i=i'$ and $j=j'$} \\
    -\frac{1}{n^2(n+1)} & \text{if $i=i'$ and $j\neq j'$}.
  \end{cases}
  $$
  Moreover $\bM_{i,j}$ and $\bM_{i',j'}$ are independent if and only if $i\neq
  i'$.
\end{cor}

The set $\cM_n$ is also a compact semi--group for the matrix product. The
following two theorems concern the translation invariance of $\cU(\cM_n)$ and
the question of the existence of an idempotent probability distribution on
$\cM_n$.

\begin{thm}[Translation invariance]\label{th:perm-trans}
  For every $\bT\in\cM_n$, the law $\cU(\cM_n)$
  is invariant by the left translation $\bM\mapsto \bT\bM$ if and only if
  $\bT$ is a permutation matrix. The same holds true for the right translation
  $\bM\mapsto \bM\bT$.
\end{thm}

\begin{thm}[Idempotent distributions]\label{th:negative}
  There is no probability distribution on $\cM_n$, absolutely continuous with
  respect to $\cU(\cM_n)$, with full support, and which is invariant by every
  left translations $\bM\mapsto\bT\bM$ where $\bT$ runs over $\cM_n$. The same
  holds true for right translations.
\end{thm}

The proofs of theorems \ref{th:DME}, \ref{th:perm-trans}, \ref{th:negative}
and corollary \ref{co:DME} are given in section \ref{se:DME}.

\subsubsection*{Asymptotic behavior of singular values and eigenvalues}

The spectral properties of large dimensional random matrices are connected to
many areas of mathematics, see for instance the books
\cite{MR2129906,MR1746976,bai-silverstein-book,anderson-guionnet-zeitouni,forrester,MR2168344}
and the survey \cite{MR1711663}. If $\bM\sim\cU(\cM_n)$, then almost surely,
the real matrix $\bM$ is invertible, non--normal, with neither independent nor
centered entries. The singular values of certain large dimensional centered
random matrices with independent rows is considered for instance in
\cite{MR2280648} and \cite{MR2265341,pajor-pastur}.

For any square $n\times n$ matrix $\bA$ with real or complex entries, let the
complex eigenvalues $\lambda_1(\bA),\ldots,\lambda_n(\bA)$ of $\bA$ be labeled
so that $\ABS{\lambda_1(\bA)}\geq\cdots\geq\ABS{\lambda_n(\bA)}$. The
\emph{spectral radius} of $\bA$ is thus given by $|\lambda_1(\bA)|=\max_{1\leq
  k\leq n}|\lambda_k(\bA)|$. The \emph{empirical spectral distribution} (ESD)
of $\bA$ is the discrete probability distribution on $\dC$ with at most $n$
atoms defined by
$$
\frac{1}{n}\sum_{k=1}^n \delta_{\lambda_k(\bA)}.
$$
The \emph{singular values} $s_1(\bA)\geq\cdots\geq s_n(\bA)\geq0$ of $\bA$ are
the eigenvalues of the positive semi--definite Hermitian matrix
$(\bA\bA^*)^{1/2}$ where
$$
\bA^*=\OL{\bA}^\top
$$
denotes the conjugate transpose of $\bA$. Namely, for every $1\leq k\leq n$,
$$
s_k(\bA)=\lambda_k(\sqrt{\bA\bA^*})=\sqrt{\lambda_k(\bA\bA^*)}.
$$
Note that $\bA\bA^*$ and $\bA^*\bA$ share the same spectrum. The atoms of the
ESD of $\sqrt{\bA\bA^*}$ are $s_1(\bA),\ldots,s_n(\bA)$. The singular values
of $\bA$ have a clear geometrical interpretation: the linear operator $\bA$
maps the unit ball to an ellipsoid, and the singular values of $\bA$ are
exactly the half--lengths of its principal axes. In particular,
$s_1(\bA)=\max_{\NRM{x}_2=1}\NRM{\bA x}_2=\NRM{\bA}_{2\to 2}$ while
$s_n(\bA)=\min_{\NRM{x}_2=1}\NRM{\bA x}_2=\NRM{\bA^{-1}}_{2\to2}^{-1}$.
Moreover, $\bA$ has exactly $\mathrm{rank}(\bA)$ non zero singular values. The
relationship between the eigenvalues and the singular values are captured by
the Weyl--Horn inequalities
$$
\forall k\in\{1,\ldots,n\},
\quad
\prod_{i=1}^k|\lambda_i(\bA)|
\leq
\prod_{i=1}^k s_i(\bA)
\quad
\text{with equality when $k=n$},
$$
see \cite{MR0061573,MR0030693}.
If $\bA$ is normal, i.e.\ $\bA\bA^*=\bA^*\bA$, then
$s_k(\bA)=\ABS{\lambda_k(\bA)}$ for every $1\leq k\leq n$. Back to our
Dirichlet Markov Ensemble, if $\bM\sim\cU(\cM_n)$ then $\bM$ is almost surely
a non--normal matrix, and thus one cannot express the singular values of $\bM$
in terms of the eigenvalues of $\bM$. The following theorem gives the
asymptotic behavior of the empirical distribution built from the singular values of $\bM$.

\begin{thm}[Singular values for Dirichlet Markov Ensemble]\label{th:sym}
  Let $(X_{i,j})_{1\leq i,j<\infty}$ be an infinite array of i.i.d.\
  exponential random variables of unit mean. For every $n$, let $\bM$ be the
  $n\times n$ random matrix defined for every $1\leq i,j\leq n$ by
  $$
  \bM_{i,j}=\frac{X_{i,j}}{\sum_{k=1}^n X_{i,k}}.
  $$
  Then $\bM\sim\cU(\cM_n)$ and 
  $$
  \dP\PAR{
    \frac{1}{n}\sum_{k=1}^n\delta_{\lambda_k(n\bM\bM^\top)}
    \overset{\text{w}}{\underset{n\to\infty}{\longrightarrow}}
    \cP_1}=1
  $$
  where $\overset{\text{w}}{\to}$ denotes the weak convergence of probability
  distributions and $\cP_1$ the Marchenko--Pastur distribution defined in table
  \ref{ta:distros}. In other words,
  $$
  \dP\PAR{
    \frac{1}{n}\sum_{k=1}^n\delta_{s_k(\sqrt{n}\,\bM)}
    \overset{\text{w}}{\underset{n\to\infty}{\longrightarrow}}
    \cQ_1}=1
  $$
  where $\cQ_1$ denotes the Wigner quarter--circle distribution 
  defined in table \ref{ta:distros}.
\end{thm}

\medskip

\begin{table}
\begin{center}
{ \footnotesize
\begin{tabular}{c|c|c}
  \textbf{Probability distribution name}
  & \textbf{Support} 
  & \textbf{Lebesgue density} 
  \\ \hline
  &
  & \\
  Circle or circular law $\cC_\sigma$
  & $\{z\in\dC;|z|\leq\sigma\}\subset\dC$
  & $z\mapsto(\pi\sigma^2)^{-1}$ \\
  &
  & \\
  Wigner semi--circle distribution $\cW_\sigma$
  & $[-2\sigma,+2\sigma]\subset\dR$ 
  & $x\mapsto (2\pi\sigma^2)^{-1} \sqrt{4\sigma^2-x^2}$ \\
  &
  & \\
  Wigner quarter--circle distribution $\cQ_\sigma$
  & $[0,2\sigma]\subset\dR$ 
  & $x\mapsto (\pi\sigma^2)^{-1} \sqrt{4\sigma^2-x^2}$ \\
  &
  & \\
  Marchenko--Pastur distribution $\cP_\sigma$
  & $[0,4\sigma^2]\subset\dR$ 
  & $x\mapsto (2\pi \sigma^2 x)^{-1} \sqrt{x(4\sigma^2-x)}$ 
\end{tabular}
}
\caption{Some of the remarkable probability distributions in random matrices.}
\label{ta:distros}
\end{center}
\end{table}

\medskip

Following the notations of table \ref{ta:distros}, for every real fixed
parameter $\sigma>0$, every real random variable $W$, and every complex random
variable $Z=U+\sqrt{-1}V$ with $U=\mathrm{RealPart}(Z)$ and
$V=\mathrm{ImaginaryPart}(Z)$, we have, by a change of variables,
$$
\PAR{W^2\sim\cP_\sigma\,\Leftrightarrow\,|W|\sim\cQ_\sigma}
\quad\text{and}\quad
\PAR{W\sim\cW_\sigma\,\Rightarrow\,W^2\sim\cP_\sigma\ \text{and}\ |W|\sim\cQ_\sigma}.
$$
Moreover, we have, simply by using the Cram\'er-Wold theorem,
$$
Z\sim\cC_{2\sigma}\ \Leftrightarrow\ 
\PAR{\mathrm{RealPart}(e^{\sqrt{-1}\theta}Z)\sim \cW_\sigma
\ \text{for every}\ \theta\in[0,2\pi)}.
$$
In particular, we have
$$
Z\sim\cC_{2\sigma}\ \Rightarrow\ U\sim\cW_\sigma\ \text{and}\ V\sim\cW_\sigma.
$$
Beware however that $U$ and $V$ are not independent random variables!
Furthermore, if $\dP(|Z|=\sigma;V\geq0)=1$ then $Z$ follows the uniform
distribution over the upper half circle of radius $\sigma$ if and only if $U$
follows the so--called arc--sine distribution on $[-\sigma,+\sigma]\subset\dR$
with Lebesgue density $x\mapsto (\pi\sqrt{\sigma^2-x^2})^{-1}$.

The proof of theorem \ref{th:sym} is given in section \ref{se:proofs}. Since
$\ABS{\lambda_1(\bA)}\leq s_1(\bA)$ for any square matrix $\bA$, and since
$\lambda_1(\bM)=1$, we have for every $n\geq1$
$$
s_1(\bM)\geq \ABS{\lambda_1(\bM)}=1.
$$
However, theorem \ref{th:sym} implies in particular that almost surely
$$
\frac{1}{n}\,\mathrm{Card}\BRA{1\leq k\leq n\text{ such that }
s_k(\bM)>\frac{2}{\sqrt{n}}} \underset{n\to\infty}{\longrightarrow}0.
$$

\subsubsection*{Random $Q$-matrices}

Bryc, Dembo, and Jiang studied in \cite{MR2206341} the limiting spectral
distribution of \emph{random Hankel, Markov, and Toeplitz matrices}. Let us
explain briefly what they mean by ``random Markov matrices''. They proved the
following theorem (see \cite[th. 1.3]{MR2206341} and also \cite{MR2399292}) :
let $(\bX_{i,j})_{1<i<j<\infty}$ be an infinite triangular array of i.i.d.\
real random variables of mean $0$ and variance $1$. Let $\bQ$ be the symmetric
$n\times n$ random matrix defined for every $1\leq i\leq j\leq n$ by
$\bQ_{i,j}=\bQ_{j,i}=X_{i,j}$ if $i<j$, and
$$
\bQ_{i,i}=
-\ds{\sum_{\SSK{1\leq k\leq n\\k\neq i}}} \bQ_{i,k} 
\quad\text{for every $1\leq i\leq n$.}
$$
Then, almost surely, the ESD of $n^{-1/2}\,\bQ$ converges as $n\to\infty$ to the
free convolution\footnote{This limiting spectral distribution is a symmetric
  law on $\dR$ with smooth bounded density of unbounded support. See
  \cite{MR1746976} or \cite{MR1488333} for Voiculescu's free convolution.} of
a semi--circle law and a standard Gaussian law.

This result gives an answer to a precise question raised by Bai in his 1999
review article \cite[sec. 6.1.1]{MR1711663}. The matrix $\bQ$ is not Markov.
However, it looks like a Markov generator, i.e.\ a $Q$-matrix, since its rows
sum up to $0$. Unfortunately, the assumptions do not allow the off--diagonal
entries of $\bQ$ to have non--negative support, and thus $\bQ$ cannot be almost
surely a Markov generator. In particular, if $\bI$ stands for the identity
matrix of size $n\times n$, the symmetric matrix $\bM=\bQ+\bI$ cannot be
almost surely Markov.

\subsubsection*{Eigenvalues and the circular law} 

If $\bM$ is as in theorem \ref{th:sym}, then $\lambda_1(\sqrt{n}\,\bM)=\sqrt{n}$
goes to $+\infty$ as $n\to\infty$ while its weight in the ESD is $1/n$. Thus,
it does not contribute to the limiting spectral distribution of $\sqrt{n}\,\bM$.
Numerical simulations (see figure \ref{fi:cirlaw-DME}) suggest that the
empirical distribution of the rest of the spectrum tends as $n\to\infty$ to
the uniform distribution on the unit disc. One can formulate this conjecture
as follows.

\begin{conj}[Circle law for the Dirichlet Markov Ensemble]\label{cj:cirlaw}
  If $\bM$ is as in theorem \ref{th:sym}, then
  $$
  \dP\PAR{\frac{1}{n}\sum_{k=1}^n\delta_{\lambda_k(\sqrt{n}\,\bM)}
    \overset{\text{w}}{\underset{n\to\infty}{\longrightarrow}}
    \cC_1}=1
  $$
  where $\overset{\text{w}}{\to}$ denotes the weak convergence of probability
  distributions and $\cC_1$ the uniform distribution over the unit disc
  $\{z\in\dC;\ABS{z}\leq1\}$ as defined in table \ref{ta:distros}.
\end{conj}

The main difficulty in conjecture \ref{cj:cirlaw} lies in the fact that $\bM$
is non--normal with non i.i.d.\ entries. The limiting spectral distributions of
non--normal random matrices is a notoriously difficult subject, see for
instance \cite{tao-vu-krishnapur-circular}. The method used for the singular
values for the proof of theorem \ref{th:sym} fails for the eigenvalues, due to
the lack of variational formulas for the eigenvalues. In contrast to singular
values, the eigenvalues of non--normal matrices are very sensitive to
perturbations, a phenomenon captured by the notion of pseudo--spectrum
\cite{MR2155029}. The reader may find in \cite{chafai-cir} a more general
version of theorem \ref{th:MP} which goes beyond the exponential case, and
some partial answers to conjecture \ref{cj:cirlaw}.

\subsubsection*{Sub--dominant eigenvalue}

The fact that non--centered entries produce an explosive extremal eigenvalue
was already noticed in various situations, see for instance \cite{MR1062321},
\cite{MR1284550}, \cite[th. 1.4]{MR2206341}, \cite{MR2284045}, and
\cite{djalil-nccl}. It is natural to ask about the asymptotic behavior
(convergence and fluctuations) of the sub--dominant eigenvalue
$\lambda_2(\bM)$ when $\bM\sim\cU(\cM_n)$. The reader may find some answers in
\cite{MR1972678,MR1800967}, and may forge new conjectures from our simulations
(see figures \ref{fi:subdom1} and \ref{fi:subdom2}). For instance, by analogy
with the Complex Ginibre Ensemble \cite{MR1148410,MR1986426}, one can state
the following:

\begin{conj}[Behavior of sub--dominant eigenvalue and spectral gap]
  If $\bM$ is as in theorem \ref{th:sym}, then $\lambda_1(\bM)=1$ while
  $$
  \dP\PAR{\lim_{n\to\infty}\sqrt{n}\,\ABS{\lambda_2(\bM)}=1}=1.
  $$
  In particular, the spectral gap $1-|\lambda_2(\bM)|$ of $\bM$ is of order
  $1-1/\sqrt{n}$ for large $n$. Moreover, there exist deterministic sequences
  $(a_n)$ and $(b_n)$ and a probability distribution $\cG$ on $\dR$ such that
  $$
  b_n(\ABS{\lambda_2(\bM)}-a_n)
  \overset{\text{d}}{\underset{n\to\infty}{\longrightarrow}}
  \cG
  $$
  where $\overset{\text{d}}{\to}$ denotes the convergence in law.
\end{conj}

There is not clear indication that $\cG$ is a Gumbel distribution as for the
Complex Ginibre Ensemble. Moreover, our simulations suggest that the
sub--dominant eigenvalue is real with positive probability (depends on $n$),
which is not surprising knowing \cite{MR1437734,MR1231689}. Note that Goldberg
and Neumann have shown \cite{MR1972678} that if $\bX$ is an $n\times n$ random
matrix with i.i.d.\ rows such that for every $1\leq i,j,j'\leq n$,
$$
\dE[\bX_{i,j}]=\frac{1}{n}, %
\quad\text{and}\quad %
\mathrm{Var}(\bX_{i,j})=O\PAR{\frac{1}{n^2}}, %
\quad\text{and}\quad %
\ABS{\mathrm{Cov}(\bX_{i,j},\bX_{i,j'})}=O\PAR{\frac{1}{n^3}}
$$
then $\dP(\ABS{\lambda_2(\bX)}\leq r)\geq p$ for any $p\in(0,1)$, any $0<r<1$,
and large enough $n$. This is the case if we set $\bX=\bM$.

\subsubsection*{Other distributions}

The Dirichlet distribution of dimension $n$ and mean
$(\frac{1}{n},\ldots,\frac{1}{n})$ is the uniform distribution on the simplex
$\Lambda_n$ defined by \eqref{eq:simplex}. One can replace the uniform
distribution by a Dirichlet distribution of dimension $n$ and arbitrary mean.
The argument used in the proof of theorem \ref{th:sym} remains the same due to
the very similar construction of Dirichlet distributions by projection from
i.i.d.\ Gamma random variables. One can also replace the $\NRM{\cdot}_1$-norm
by any other $\NRM{\cdot}_p$-norm, and investigate the limiting spectral
distribution of the corresponding random matrices. This case can be handled
with the construction of the uniform distribution by projection proposed in
\cite{schechtman-zinn}. Replacing the non--negative portion of spheres by the
non--negative portion of balls is also possible by using \cite{barthe-al}.
More generally, one can consider random matrices with independent rows. The
case of the uniform distribution on the whole unit $\NRM{\cdot}_p$--ball of
$\dR^n$ is considered for instance by in \cite{MR2280648} by using
\cite{barthe-al} together with random matrices results for i.i.d.\ centered
entries. It is crucial here to have an explicit construction of the
distribution from an i.i.d.\ array. For the link with the sampling of convex
bodies, see \cite{MR2276637}. The case of matrices with i.i.d.\ rows following
a log-concave isotropic distribution is considered in the recent work
\cite{pajor-pastur}, by using recently developed results on log-concave
measures. The reader may find a universal version of theorem \ref{th:MP} in
\cite{chafai-cir}, where the exponential law is replaced by an arbitrary law.

\subsubsection*{Doubly Stochastic matrices}

The Birkhoff or transportation polytope is the set of $n\times n$ doubly
stochastic matrices, i.e.\ matrices which are Markov and have a Markov
transpose. Each $n\times n$ doubly stochastic matrix corresponds to a
transportation map of $n$ unit masses into $n$ boxes of unit mass (matching),
and conversely, each transportation map of this kind is a $n\times n$ doubly
stochastic matrix. Geometrically, the Birkhoff polytope is a convex compact
subset of $\cM_n$ of zero Lebesgue measure in $\dR^{n^2}$ and $(n-1)^2$
degrees of freedom if $n>1$. As for $\cM_n$, one can define the uniform
distribution as the normalized trace of the Lebesgue measure. However, we
ignore if this distribution has a probabilistic representation that allows
exact simulation as for $\cU(\cM_n)$. The spectral behavior of random doubly
stochastic matrices was considered in the Physics literature, see for instance
\cite{MR1840293}. On the purely discrete side, the Birkhoff polytope is also
related to magic squares, transportation polytopes and contingency tables, see
\cite{MR901121,MR803747} and \cite{MR1380519}. Note also that if $\bM$ is
Markov, then $\bM\bM^\top$ and $\frac{1}{2}(\bM+\bM^\top)$ are not Markov in
general. However, this is the case when $\bM$ is doubly stochastic. The
Birkhoff-von Neumann theorem states that the extremal points of the Birkhoff
polytope are exactly the permutation matrices. The reader may find nice
spectral results on random uniform permutation matrices in
\cite{MR1794543,MR1813834} and references therein.

Another interesting polytope of matrices is the set of symmetric $n\times n$
Markov matrices, which is a convex compact polytope of zero Lebesgue measure
in $\dR^{n^2}$ with $\frac{1}{2}n(n-1)$ degrees of freedom if $n>1$. As for
$\cM_n$, one can define the uniform distribution as the normalized trace of
the Lebesgue measure. However, we ignore if this distribution has a
probabilistic representation that allows simulation as for $\cU(\cM_n)$. One
can ask about the spectral properties of the corresponding random symmetric
Markov matrices. Note that these matrices are doubly stochastic, but the
converse is false except when $n=1$ or $n=2$. 
Our construction of $\cU(\cM_n)$ in theorem \ref{th:sym} corresponds
in the Markovian probabilistic jargon to a random conductance model
on the complete oriented graph. The study of the spectral
properties of random reversible Markov conductance models on the complete non--oriented graph can be found in
\cite{chafai-stoc,bordenave-caputo-chafai,bordenave-caputo-chafai-bis}.
For other graphs, the reader may find some clues in \cite{MR2166276}.

Let $\bM$ be as in theorem \ref{th:sym}. Numerical simulations suggest that
almost surely, the ESD of the symmetric matrix $\frac{1}{2}(\bM+\bM^\top)$
tends, as $n\to\infty$, to a semi--circle Wigner distribution.

If $\bU$ is an $n\times n$ unitary matrix, then $(\ABS{\bU_{i,j}}^2)_{1\leq
  i,j\leq n}$ is a doubly stochastic matrix. These doubly stochastic matrices
are called \emph{uni--stochastic} or \emph{unitary-stochastic}. There exists
doubly stochastic matrices which are not uni--stochastic, see \cite{MR2172684}
and \cite{MR1876609}. However, every permutation matrix is orthogonal and thus
uni--stochastic. The Haar measure on the unitary group induces a probability
distribution on the set of uni--stochastic matrices. How about the asymptotic
spectral properties of the corresponding random matrices?

\subsubsection*{Perron--Frobenius eigenvector (invariant vector)}

If $\bM\sim\cU(\cM_n)$, then almost surely, all the entries of $\bM$ are
non-zero, and in particular, $\bM$ is almost surely recurrent irreducible and
aperiodic. By a theorem of Perron and Frobenius \cite{MR2209438}, it follows
that almost surely, the eigenspace of $\bM^\top$ associated to the eigenvalue
$1$ is of dimension $1$ and contains a unique vector with non--negative entries
and unit $\NRM{\cdot}_1$-norm. One can ask about the asymptotic behavior of
this vector as $n\to\infty$. For a fixed $n$, the distribution of this vector
is the distribution of the rows of the infinite product of random matrices
$\lim_{k\to\infty}\,\bM^k$.

\section{Structure of the Dirichlet Markov Ensemble}
\label{se:DME}

Let $\Lambda_n$ be as in \eqref{eq:simplex}. For any $a\in(0,\infty)^n$, the
Dirichlet distribution $\cD_n(a_1,\ldots,a_n)$, supported by $\Lambda_n$, is
defined as the distribution of
$$ 
\frac{1}{\NRM{G}_1}G %
=\PAR{\frac{G_1}{G_1+\cdots+G_n},\ldots,\frac{G_n}{G_1+\cdots+G_n}}
$$ where $G$ is a random vector of $\dR^n$ with independent entries with
$G_i\sim\mathrm{Gamma}(1,a_i)$ for every $1\leq i\leq n$. Here
$\mathrm{Gamma}(\lambda,a)$ has density
$$
t\mapsto \frac{\lambda^a}{\Gamma(a)} t^{a-1}e^{-\lambda t}\,\rI_{(0,\infty)}(t),
$$
where $\Gamma(a)=\int_0^\infty\!t^{a-1}e^{-t}\,dt$ is the Euler Gamma function.
Let $P\sim\cD_n(a_1,\ldots,a_n)$. For every partition $I_1,\ldots,I_k$ of
$\{1,\ldots,n\}$ into $k$ non empty subsets, we have
$$
\PAR{\sum_{i\in I_1}P_i,\ldots,\sum_{i\in I_k} P_i}
\sim\cD_k\PAR{\sum_{i\in I_1}a_i,\ldots,\sum_{i\in I_k}a_i}.
$$
The mean and covariance matrix of $\cD_n(a_1,\ldots,a_n)$ are given by
$$
\frac{1}{\NRM{a}_1}a %
\quad\text{and}\quad %
\frac{1}{\NRM{a}_1^2(1+\NRM{a}_1)} (\NRM{a}_1\mathrm{diag}(a)-aa^\top)
$$
where $a=(a_1,\ldots,a_n)^\top$ and $\mathrm{diag}(a)$ is the diagonal matrix
with diagonal given by $a$. For any non-empty subset $I$ of $\{1,\ldots,n\}$,
we have
$$
\sum_{i\in I} P_i %
\sim \mathrm{Beta}\PAR{\sum_{i\in I}a_i,\sum_{i\not\in I}a_i},
$$
where $\mathrm{Beta}(\alpha,\beta)$ denotes the Euler Beta distribution on $[0,1]$
of Lebesgue density
$$
t\mapsto \frac{\Gamma(\alpha+\beta)}{\Gamma(\alpha)\Gamma(\beta)} 
t^{\alpha-1}(1-t)^{\beta-1}\,\rI_{[0,1]}(t).
$$
If $P_I=(P_i)_{i\in I}$, $P_{I^c}=(P_i)_{i\not\in I}$, $a_I=(a_i)_{i\in I}$,
and $\ABS{I}=\mathrm{card}(I)$, then
$$
\frac{1}{\sum_{i\in I}P_i}P_I \text{ and } P_{I^c} \text{ are independent and
} \frac{1}{\sum_{i\in I}P_i}P_I\sim\cD_{\ABS{I}}(a_I),
$$
For any $\alpha>0$, the Dirichlet distribution $\cD_n(\alpha,\ldots,\alpha)$
is exchangeable, with negatively correlated components. More generally, if
$P\sim\mu$ where $\mu$ is an exchangeable probability distribution on the
simplex $\Lambda_n$ with $n>1$, then
$$
0=\mathrm{Var}(1)=\mathrm{Var}(P_1+\cdots+P_n)%
=n\mathrm{Var}(P_1)+n(n-1)\mathrm{Cov}(P_1,P_2).
$$
Consequently, $\mathrm{Cov}(P_1,P_2)=-(n-1)^{-1}\mathrm{Var}(P_1)$ and in
particular $\mathrm{Cov}(P_1,P_2)\leq0$. 

We refer for instance to \cite{MR0144404} for other properties of Dirichlet
distributions. Corollary \ref{co:DME} follows immediately from theorem
\ref{th:DME} together with the basic properties of the Dirichlet distributions
mentioned above.

\begin{proof}[Proof of theorem \ref{th:DME}]
  As a subset of $\dR^n$, the simplex $\Lambda_n$ defined by
  \eqref{eq:simplex} is of zero Lebesgue measure. However, by considering
  $\Lambda_n$ as a convex subset of the hyper-plane of equation
  $x_1+\cdots+x_n=1$ or by using the general notion of Hausdorff measure, one
  can see that in fact, the Dirichlet distribution $\cD_n(1,\ldots,1)$ is the
  normalized trace of the Lebesgue measure of $\dR^n$ on the simplex
  $\Lambda_n$. In other words, $\cD_n(1,\ldots,1)$ can be seen as the uniform
  distribution on $\Lambda_n$, see \cite{schechtman-zinn}.

  We identify $\cM_n$ with
  $(\Lambda_n)^n=\Lambda_n\times\cdots\times\Lambda_n$ where $\Lambda_n$ is
  repeated $n$ times. The trace of the Lebesgue measure of
  $\dR^{n^2}=(\dR^n)^n$ on $(\Lambda_n)^n$ is the $n$-tensor product of the
  trace of the Lebesgue measure of $\dR^n$ on $\Lambda_n$, i.e.\ the
  $n$-tensor product measure $\cD_n(1,\ldots,1)^{\otimes n}$. Consequently,
  for every positive integer $n$,
  $$
  (\cM_n,\cU(\cM_n))=((\Lambda_n)^n,\cD_n(1,\ldots,1)^{\otimes n}).
  $$
  This gives the invariance of $\cU(\cM_n)$ by permutation of rows. If
  $\bM\sim\cU(\cM_n)$, then the rows of $\bM$ are i.i.d.\ and follow the
  Dirichlet distribution $\cD_n(1,\ldots,1)$. Finally, the invariance of
  $\cU(\cM_n)$ by permutation of columns comes from the exchangeability of the
  Dirichlet distribution $\cD_n(1,\ldots,1)$.
\end{proof}

\subsubsection*{Recursive simulation}

The simulation of $\cU(\cM_n)$ follows from the simulation of $n$ i.i.d.\
realizations of $\cD_n(1,\ldots,1)$ by using $n^2$ i.i.d.\ exponential random
variables. The elements of Dyson's classical Gaussian ensembles GUE and GOE
can be simulated recursively by adding a new independent line/column. It is
natural to ask about a recursive method for the Dirichlet Markov Ensemble. If
$$
X\sim\cD_{n-1}(a_2,\ldots,a_n)%
\quad\text{and}\quad %
Y\sim\mathrm{Beta}(a_1,a_2+\cdots+a_n)
$$
are independent, then
$$
(Y,(1-Y)X)\sim\cD_n(a_1,\ldots,a_n).
$$
This recursive simulation of Dirichlet distributions is known as the
\emph{stick--breaking} algorithm \cite{MR1309433}. It allows to simulate
$\cU(\cM_n)$ recursively on $n$. Namely, if $\bM$ is such that
$\bM\sim\cU(\cM_n)$, then
$$
\left(
\begin{array}{cc}
Y &  (1-Y)\cdot \bM \\
Z_1 & Z_2\ \cdots\ Z_n 
\end{array}
\right)\sim\cU(\cM_{n+1})
$$
where $Z$ is a random row vector of $\dR^{n+1}$ with
$Z\sim\cD_{n+1}(1,\ldots,1)$ and $Y$ is a random column vector of $\dR^n$ with
i.i.d.\ entries of law $\mathrm{Beta}(1,n)$, with $\bM,Y,Z$ independent. Here
$((1-Y)\cdot\bM)_{i,j}:=(1-Y)_i\bM_{i,j}$ for every $1\leq i,j\leq n$.

\subsubsection*{Asymptotic behavior of the rows}

Let $\bM$ and $(X_{i,j})_{1\leq i,j<\infty}$ be as in theorem \ref{th:sym}.
Let us fix $k\geq1$ and $n\geq i\geq1$. The $k^\text{th}$ moment $m_{n,i,k}$
of the discrete probability distribution
$\frac{1}{n}\sum_{j=1}^n\delta_{n\bM_{i,j}}$ is given by
\begin{align*}
m_{n,i,k}&=\frac{1}{n}\sum_{j=1}^n \PAR{n\bM_{i,j}}^k \\
&=\sum_{j=1}^n \frac{n^k}{n}\frac{X_{i,j}^k}{\PAR{X_{i,1}+\cdots+X_{i,n}}^k} \\
&= \frac{n^k}{\PAR{X_{i,1}+\cdots+X_{i,n}}^k}\frac{X_{i,1}^k+\cdots+X_{i,n}^k}{n}.
\end{align*}
Therefore, by using twice the strong law of large numbers, we get that almost
surely,
$$
\lim_{n\to\infty} m_{n,i,k} = \frac{\dE[X_{1,1}^k]}{\dE[X_{1,1}]^k}=\dE[X_{1,1}^k].
$$
As a consequence, almost surely, for any fixed $i\geq1$ and every $k\geq1$,
$$
\lim_{n\to\infty}W_k\PAR{\frac{1}{n}\sum_{j=1}^n\delta_{n\bM_{i,j}}\,;\,\cE_1}=0,
$$
where $\cE_1=\cL(X_{1,1})$ is the exponential law on unit mean and where
$W_k(\cdot\,;\,\cdot)$ is the so called Wasserstein--Mallows coupling distance
of order $k$ (see for instance \cite{MR1964483} or \cite{MR1105086}). This
result is a special case of a more general well known phenomenon (sometimes
referred as the Poincar\'e observation) concerning the coordinates of a
uniformly distributed random point on the unit $\NRM{\cdot}_p$--sphere of
$\dR^n$ with $1\leq p<\infty$ when $n\to\infty$, see for instance
\cite{MR1962135}, \cite{MR2480790}, and references therein.

\subsubsection*{Semi--group structure and translation invariance}

The set $\cM_n$ is a semi--group for the usual matrix product. In particular,
for every $\bT\in\cM_n$, the set $\cM_n$ is stable by the left translation
$\bM\mapsto \bT\bM$ and the right translation $\bM\mapsto \bM\bT$. When $\bT$
is a permutation matrix, then these translations are bijective maps, and the
left translation (respectively right) translation corresponds to rows
(respectively columns) permutations.

For some fixed $\bT\in\cM_n$, let us consider the left translation $\bM\mapsto
\bT\bM$, where $\bM\sim\cU(\cM_n)$. By linearity, we have
$$
\dE[\bT\bM] %
= \bT\dE[\bM] %
= \bT\frac{1}{n}\mathbf{1} %
= \frac{1}{n}\mathbf{1}
$$
where $\mathbf{1}$ is the $n\times n$ matrix full of ones. Thus, the left
translation by $\bT$ leaves the mean invariant.

\begin{proof}[Proof of theorem \ref{th:perm-trans}]
  First of all, the case $n=1$ is trivial and one can assume that $n>1$ in the
  rest of the proof. A probability distribution $\mu$ on $\cM_n$ is invariant
  by the left translation $\bM\mapsto \bP\bM$ for every permutation matrix
  $\bP$ of size $n\times n$ if and only if $\mu$ is row exchangeable.
  Similarly, $\mu$ is invariant by the right translation $\bM\mapsto \bM\bP$
  for every permutation matrix $\bP$ of size $n\times n$ if and only if $\mu$
  is column exchangeable. Theorem \ref{th:DME} gives then the invariance of
  $\cU(\cM_n)$ by left and right translations with respect to permutation
  matrices\footnote{However that as a law over $\dR^{n^2}$, $\cU(\cM_n)$ is
    not exchangeable. The permutation of rows and columns correspond to a
    proper subset of the group of permutations of the $n^2$ entries.}.

  Conversely, let us assume that the law $\cU(\cM_n)$ is invariant by the left
  translation $\bM\mapsto \bT\bM$ for some $\bT\in\cM_n$. If
  $\bM\sim\cU(\cM_n)$, and since the components of the first column
  $\bM_{\cdot,1}$ of $\bM$ are i.i.d.\ we have
  \begin{align*}
    \mathrm{Var}((\bT\bM)_{1,1})
    &=\mathrm{Var}\PAR{\sum_{k=1}^n\bT_{1,k}\bM_{k,1}} \\
    &=\sum_{k=1}^n(\bT_{1,k})^2\mathrm{Var}\PAR{\bM_{k,1}} \\
    &=\mathrm{Var}(\bM_{1,1})\sum_{k=1}^n(\bT_{1,k})^2.
  \end{align*}
  The invariance hypothesis implies in particular that
  $\mathrm{Var}(\bM_{1,1})=\mathrm{Var}((\bT\bM)_{1,1})$. Since
  $\mathrm{Var}(\bM_{1,1})=(n-1)/(n^2(n+1))>0$, we get
  $1=\sum_{k=1}^n(\bT_{1,k})^2$. Now, $\bT$ is Markov and thus
  $\sum_{k=1}^n\bT_{1,k}=1$, which gives
  $$
  \sum_{k=1}^n(\bT_{1,k}-(\bT_{1,k})^2)=0.
  $$
  Since $\bT$ is Markov, its entries are in $[0,1]$ and hence
  $\bT_{1,k}\in\{0,1\}$ for every $1\leq k\leq n$. The condition $\sum_{k=1}^n
  \bT_{1,k}=1$ gives then that the first line of $\bT$ is an element of the
  canonical basis of $\dR^n$. The same argument used for $(\bT\bM)_{k,1}$ for
  every $1\leq k\leq n$ shows that every line of $\bT$ is an element of the
  canonical basis, and thus $\bT$ is a binary matrix with exactly a unique $1$
  on each row. Since $\bT\bM\sim\cU(\cM_n)$, it has independent rows, and thus
  the position of the $1$'s on the rows of $\bT$ are pairwise different, which
  means that $\bT$ is a permutation matrix as expected.

  Let us consider now the case where the law $\cU(\cM_n)$ is invariant by the
  right translation $\bM\mapsto \bM\bT$ for some $\bT\in\cM_n$. If
  $\bM\sim\cU(\cM_n)$, we can first take a look at the mean. Namely,
  $\dE[\bM\bT]=\dE[\bM]\bT=\frac{1}{n}\,\bS$ where $\bS$ is defined by
  $$
  \bS_{i,j}=\sum_{k=1}^n\bT_{k,j}
  $$
  for every $1\leq i,j\leq n$. Now, the invariance hypothesis gives on the
  other hand 
  $$
  \dE[\bM\bT]=\dE[\bM]=\frac{1}{n}\mathbf{1}
  $$
  and thus $\bS=\mathbf{1}$, which means that $\bT$ is doubly stochastic, i.e.\ 
  both $\bT$ and $\bT^\top$ are Markov. The invariance hypothesis implies also
  that
  $$
  \mathrm{Var}((\bM\bT)_{1,1})=\mathrm{Var}(\bM_{1,1})=\frac{n-1}{n^2(n+1)}.
  $$
  But since the first line $\bM_{1,\cdot}$ of $\bM$ is $\cD_n(1,\ldots,1)$
  distributed,
  \begin{align*}
    \mathrm{Var}((\bM\bT)_{1,1})
    &=\sum_{1\leq i,j\leq n}\bT_{i,1}\bT_{j,1}\mathrm{Cov}(\bM_{1,i};\bM_{1,j}) \\
    &=\frac{n-1}{n^2(1+n)}\sum_{i=1}^n(\bT_{i,1})^2-\frac{2}{n^2(n+1)}\sum_{1\leq
      i<j\leq n}\bT_{i,1}\bT_{j,1}.
  \end{align*}
  Since $\bT$ is doubly stochastic, we have $1=\sum_{i=1}^n\bT_{i,1}$ and thus
  $$
  (n-1)\sum_{i=1}^n(\bT_{i,1}-(\bT_{i,1})^2)
  =-2\sum_{1\leq i<j\leq n}\bT_{i,1}\bT_{j,1}.
  $$
  The terms of the left and right hand side have opposite signs, which gives
  that $\bT_{i,1}\in\{0,1\}$ for every $1\leq i\leq n$. The same method used
  for $(\bM\bT)_{1,k}$ for every $1\leq k\leq n$ shows that $\bT$ is a binary
  matrix. Since $\bT$ is doubly stochastic, it follows that $\bT$ is actually
  a permutation matrix, as expected.
\end{proof}

The set of $n\times n$ permutation matrices is a discrete subgroup of the
orthogonal group of $\dR^n$, isomorphic to the symmetric group $\Sigma_n$. The
group of permutation matrices plays for the Dirichlet Markov Ensemble the role
played by the orthogonal group for Dyson's GOE or COE, and the role played by
the unitary group for Dyson's GUE or CUE. In some sense, we replaced an $L^2$
Gaussian structure by an $L^1$ Dirichlet structure while maintaining the
permutation invariance.

A very natural question is to ask about the existence of a convolution
idempotent probability distribution on the compact semi--group $\cM_n$. Recall
that a probability distribution $\mu$ on a semi--group $\mathfrak{S}$ is
idempotent if and only if $\mu*\mu=\mu$. Here the convolution $\mu*\nu$ of two
probability distributions $\mu$ and $\nu$ on $\mathfrak{S}$ is defined, for
every bounded continuous $f:\mathfrak{S}\to\dR$, by
$$
\int_{\mathfrak{S}}\!f(s)d(\mu*\nu)(s) %
=
\int_{\mathfrak{S}}\!\PAR{\int_{\mathfrak{S}}\!f(s_ls_r)\,d\mu(s_l)}\,d\nu(s_r).
$$
Actually, the structure of compact semi--groups and their idempotent measures
was deeply investigated in the 1960's, see \cite[p. 158-160]{MR0329037} for a
historical account. In particular, one can find in \cite[lem. 3]{MR0329037}
the following result.

\begin{lem}\label{le:rosenblatt}
  Let $\mu$ be a regular probability distribution over a compact Hausdorff
  semi--group $\mathfrak{S}$ such that the support of $\mu$ generates
  $\mathfrak{S}$. Then the mass of the convolution sequence $\mu^{*n}$
  concentrates on the kernel $K(\mathfrak{S})$ of $\mathfrak{S}$. More
  precisely, for every open set $O$ containing $K$ and every $\varepsilon>0$,
  there exists a positive integer $n_\varepsilon$ such that $\mu^{*n}(O)>1-\varepsilon$
  for every $n\geq n_\varepsilon$.
\end{lem}

Here $\mu^{*n}$ denotes the convolution product $\mu*\cdots*\mu$ of $n$ copies
of $\mu$. If $\mu^{*n}$ tends to $\mu$ as $n\to\infty$ then $\mu$ is
convolution idempotent, that is $\mu*\mu=\mu$. The kernel $K(\mathfrak{S})$ of
$\mathfrak{S}$ is the sub--semi--group of $\mathfrak{S}$ obtained by taking
the intersection of the family of two sided ideals of $\mathfrak{S}$, see
\cite[th. 1]{MR0329037}. A direct consequence of lemma \ref{le:rosenblatt} is
the absence of a translation invariant probability measure $\mu$ on
$\mathfrak{S}$ with full support such that the kernel of $\mathfrak{S}$ is a
$\mu$--proper sub--semi--group of $\mathfrak{S}$. By $\mu$--proper
sub--semi--group here we mean that its $\mu$-measure is $<1$. This result can
be easily understood intuitively since the translation associated to a
non invertible element of $\mathfrak{S}$ gives a strict contraction of the
support.

\begin{proof}[Proof of theorem \ref{th:negative}]
  The kernel of the semi--group $\cM_n$ is constituted by the $n\times n$
  Markov matrices with equal rows, which are the $n\times n$ idempotent Markov
  matrices (i.e.\ $\bM^2=\bM$). The reader may find more details in \cite[p.
  146]{MR0329037}. Since the kernel of $\cM_n$ is a $\cU(\cM_n)$--proper
  sub--semi--group of $\cM_n$, lemma \ref{le:rosenblatt} implies the absence
  of any convolution idempotent probability distribution on $\cM_n$,
  absolutely continuous with respect to $\cU(\cM_n)$ and with full support.
  The proof is finished by noticing that if a probability distribution on
  $\cM_n$ is invariant by every left (or right) translation, then it is
  convolution idempotent. Note by the way that the Wedderburn matrix
  $\frac{1}{n}\mathbf{1}$ belongs to the kernel of $\cM_n$, and also that this
  kernel is equal to $\{\lim_{k\to\infty}\bM^k;\bM\in\cA_n\}$ where $\cA_n$ is
  the collection of irreducible aperiodic elements of $\cM_n$. The reader may
  find in \cite[ch. 5]{MR0329037} the structure of non fully supported
  idempotent probability distributions on compact semi--groups and in
  particular on $\cM_n$.
\end{proof}

\section{Proofs of theorem  \ref{th:sym}}
\label{se:proofs}

The following theorem can be found for instance in \cite[th.
3.6]{bai-silverstein-book}.

\begin{thm}[Singular values of large dimensional non--centered random
  arrays]\label{th:MP}
  Let $(X_{i,j})_{1\leq i,j<\infty}$ be an infinite array of i.i.d.\ real
  random variables with mean $m$ and variance $\sigma^2\in(0,\infty)$. If
  $\bX=(X_{i,j})_{1\leq i,j\leq n}$, then
  $$
  \dP\PAR{ \frac{1}{n}\sum_{k=1}^n\delta_{s_k(n^{-1/2}\bX)}
    \overset{\text{w}}{\underset{n\to\infty}{\longrightarrow}} \cQ_\sigma }=1
  $$
  where $\overset{\text{w}}{\to}$ denotes the weak converge of probability
  distributions and $\cQ_\sigma$ is the Wigner quarter--circle distribution
  defined in table \ref{ta:distros}. Moreover,
  $$
  \dP\PAR{\lim_{n\to\infty}s_1(n^{-1/2}\,\bX)=2\sigma^2}=1
  \quad\text{if and only if}\quad
  \dE[\bX_{1,1}]=0\ \text{and}\ \dE[|\bX_{1,1}|^4]<\infty.
  $$
\end{thm}


The following lemma is a consequence of \cite[le. 2]{MR1235416} (see also
\cite[le. 5.13]{bai-silverstein-book}).

\begin{lem}[Uniform law of large numbers]\label{le:LLN}
  If $(X_{i,j})_{1\leq i,j<\infty}$ is an infinite array of i.i.d.\ random
  variables of mean $m$, then by denoting $S_{i,n}=\sum_{j=1}^n X_{i,j}$,
  $$  
  \max_{1\leq i\leq n}\ABS{\frac{S_{i,n}}{n}-m}
  \overset{\text{a.s.}}{\underset{n\to\infty}{\longrightarrow}}0
  $$
  and in the case where $m\neq0$, we have also
  $$
  \max_{1\leq i\leq n}\ABS{\frac{n}{S_{i,n}}-\frac{1}{m}}
  \overset{\text{a.s.}}{\underset{n\to\infty}{\longrightarrow}}0.
  $$  
\end{lem}

The following lemma is a consequence of the Courant--Fischer variational
formulas for singular values, see \cite{MR1084815}. Also, we leave the proof
to the reader.

\begin{lem}[Singular values of diagonal multiplicative perturbations]
  \label{le:sidia}
  For every $n\times n$ matrix $\bA$, every $n\times n$ diagonal matrix $\bD$,
  and every $1\leq k\leq n$,
  $$
  s_n(\bD) s_k(\bA)\leq s_k(\bD\bA)\leq s_1(\bD) s_k(\bA).
  $$
\end{lem}

We are now able to prove theorem \ref{th:sym}.

\begin{proof}[Proof of theorem \ref{th:sym}]
  We have $\bM = \bD \bE$ where $\bE=(X_{i,j})_{1\leq i,j\leq n}$ and $\bD$ is
  the $n\times n$ diagonal matrix given for every $1\leq i\leq n$ by
  $$
  \bD_{i,i}=\frac{1}{\sum_{j=1}^n X_{i,j}}.
  $$
  The fact that $\bM\sim\cU(\cM_n)$ follows immediately from theorem
  \ref{th:DME} combined with the construction of the Dirichlet distribution
  $\cD_n(1,\ldots,1)$ from i.i.d.\ exponential random variables. It remains to
  prove the convergence of the ESD of $\sqrt{n\bM\bM^\top}$ as $n\to\infty$ to
  the Wigner quarter--circle distribution $\cQ_1$. For such, we use the method
  of Aubrun \cite{MR2280648}, by replacing the unit $\NRM{\cdot}_1$--ball by
  the portion of the unit $\NRM{\cdot}_1$--sphere with non--negative
  coordinates. If suffices to show that almost surely, the discrete measure
  $\frac{1}{n}\sum_{k=2}^n\delta_{s_k(\sqrt{n}\,\bM)}$ tends weakly to the
  Wigner quarter--circle distribution $\cQ_1$.

  We first observe that $\bE$ is a rank one additive perturbation of the
  centered random matrix $\bE-\dE\bE$. Also, a standard interlacing inequality
  gives
  $$
  s_2(\bE) \leq s_1(\bE-\dE\bM).
  $$
  Now by the second part of theorem \ref{th:MP} we have
  $s_1(\bE-\dE\bE)=O(\sqrt{n})$ almost surely. Consequently,
  $s_2(n^{-1/2}\,\bE)=O(1)$ almost surely. In particular, almost surely, the
  sequence $(\frac{1}{n}\sum_{k=2}^n\delta_{s_k(n^{-1/2}\,\bE)})_{n\geq1}$
  remains in a compact set. The desired result follows then from the
  combination of the first part of theorem \ref{th:MP} with lemmas
  \ref{le:sidia} and \ref{le:LLN}. This proof does not rely on the exponential
  nature of the $X_{i,j}$'s and remains actually valid for more general laws,
  see \cite{chafai-cir}.
\end{proof}

There is no equivalent of lemma \ref{le:sidia} for the eigenvalues instead of
the singular values, and thus the method used to prove theorem \ref{th:sym}
fails for conjecture \ref{cj:cirlaw}. Note that by lemma \ref{le:LLN} used
with the exponential distribution of mean $m=1$,
$$
\NRM{n\bD - \bI}_{2\to2} %
= \max_{1\leq i\leq n}\ABS{\frac{n}{\sum_{j=1}^n X_{i,j}}-1} %
\underset{n\to\infty}{\longrightarrow} 0 \quad\text{a.s.}
$$
If $\bA$ is diagonal, then we simply have
$\NRM{\bA}_{2\to2}=s_1(\bA)=\max_{1\leq k\leq n}\ABS{\bA_{k,k}}$, and when $\bA$ is
diagonal and invertible, $\NRM{\bA^{-1}}_{2\to2}^{-1}=s_n(\bA)=\min_{1\leq k\leq
  n}\ABS{\bA_{k,k}}$. Now, by the circular law theorem for non--central random
matrices \cite{djalil-nccl}, we get that almost surely, the ESD of
$n^{-1/2}\,\bE$ converges, as $n\to\infty$, to the uniform distribution
$\cC_1$ (see table \ref{ta:distros}). It is then natural to decompose
$\sqrt{n}\,\bM$ as
$$
\sqrt{n}\,\bM %
= n\bD n^{-1/2}\,\bE %
= (n\bD-\bI) n^{-1/2}\,\bE + n^{-1/2}\,\bE.
$$
Unfortunately, since $m=1\neq 0$, we have almost surely (see
\cite{djalil-nccl})
$$
\NRM{n^{-1/2}\,\bE}_{2\to2}=s_1(n^{-1/2}\,\bE)
\underset{n\to\infty}{\longrightarrow}+\infty.
$$
This suggests that $\sqrt{n}\,\bM$ cannot be seen as a perturbation of
$n^{-1/2}\,\bE$ with a matrix of small norm. Actually, even if it was the case,
the relation between the two spectra is unknown since $\bE$ is not normal. One
can think about using logarithmic potentials to circumvent the problem. The
strength of the logarithmic potential approach is that it allows to study the
asymptotic behavior of the ESD (i.e.\ eigenvalues) of non--normal matrices via
the singular values of a family of matrices indexed by $z\in\dC$. The details
are given in \cite{djalil-nccl} for instance. The logarithmic potential of the
ESD of $\sqrt{n}\,\bM$ at point $z$ is
\begin{align*}
  U_n(z)
  &=-\frac{1}{n}\log\ABS{\det(\sqrt{n}\,\bM-z\bI)} \\
  &=-\frac{1}{n}\log\ABS{\det(n\bD)}
  -\frac{1}{n}\log\ABS{\det(n^{-1/2}\,\bE-z(n\bD)^{-1})}.
\end{align*}
Now, by lemma \ref{le:LLN},
$$
\frac{1}{n}\log\ABS{\det(n\bD)}\underset{n\to\infty}{\longrightarrow}0%
\quad\text{a.s.}
$$
By the circular law theorem for non--central random matrices
\cite{djalil-nccl} and the lower envelope theorem \cite{MR1485778}, almost
surely, for quasi-every\footnote{This means ``except on a subset of zero
  capacity'', in the sense of potential theory, see \cite{MR1485778}.}
$z\in\dC$, the quantity
$$
\liminf_{n\to\infty} -\frac{1}{n}\log\ABS{\det(n^{-1/2}\,\bE-z\bI)}
$$
is equal to the logarithmic potential at point $z$ of the uniform distribution
$\cC_1$ on the unit disc $\{z\in\dC;\ABS{z}\leq1\}$. It is thus enough to show
that almost surely, for every $z\in\dC$,
$$
\frac{1}{n}\log\ABS{\det(n^{-1/2}\,\bE-z(n\bD)^{-1})} -
\frac{1}{n}\log\ABS{\det(n^{-1/2}\,\bE-z\bI)} %
\underset{n\to\infty}{\longrightarrow}0. 
$$
Unfortunately, we ignore how to prove that. A possible alternative beyond
potential theoretic tools is to adapt the method developed in
\cite{tao-vu-krishnapur-circular} by Tao and Vu involving a ``replacement
principle''. The reader may find some progresses in \cite{chafai-cir}.

\medskip

{\footnotesize\textbf{Acknowledgements.} Part of this work was done during two
  visits to \textsc{Laboratoire Jean Dieudonn\'e} in Nice, France. The author
  would like to thank Pierre \textsc{Del Moral} and Persi \textsc{Diaconis}
  for their kind hospitality there. Many thanks to Zhidong \textsc{Bai},
  Franck \textsc{Barthe}, W{\l}od\-zi\-mierz \textsc{Bryc}, Mireille
  \textsc{Capitaine}, Delphine \textsc{F\'eral}, Michel \textsc{Ledoux}, and
  G\'erard \textsc{Letac} for exchanging some ideas on the subject. This work
  benefited from many stimulating discussions with Neil \textsc{O'Connell}
  when he visited the \textsc{Institut de Math\'ematiques de Toulouse}. }

{%
\footnotesize%
\bibliographystyle{amsalpha}%
\bibliography{dme}%
}

\vfill

{\footnotesize %
\noindent Djalil~\textsc{Chafa\"\i} \\ %
\noindent \textsc{Laboratoire d'Analyse et de Math\'ematiques Appliqu\'ees (UMR CNRS 8050)}\\
\textsc{Universit\'e Paris-Est Marne-la-Vall\'ee}\\
\textsc{5 boulevard Descartes, Champs-sur-Marne, F-77454, Cedex 2, France.} \\
\textbf{E-mail:} \texttt{djalil(at)chafai.net} \quad %
\textbf{Web:} \texttt{http://djalil.chafai.net/} 

\begin{center}
  \begin{figure}[p]
    \begin{center}
      \includegraphics[scale=0.8]{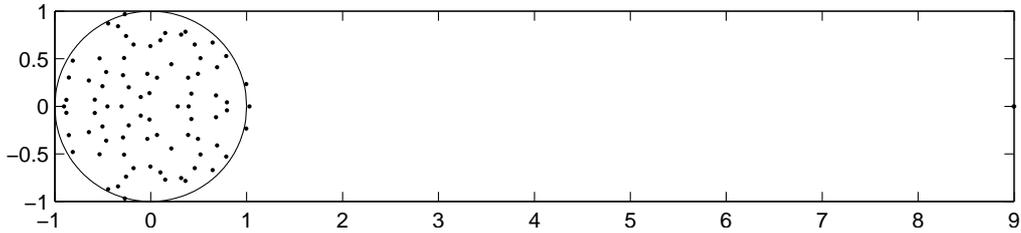}
      \caption{Plot of the spectrum of a single realization of $\sqrt{n}\,\bM$
        where $\bM\sim\cU(\cM_n)$ with $n=81$. We see one isolated
        eigenvalue $\lambda_1(\sqrt{n}\,\bM)=\sqrt{n}=9$ while the rest of the
        spectrum remains near the unit disc and seems uniformly distributed,
        in accordance with conjecture \ref{cj:cirlaw}.}
      \label{fi:cirlaw-DME}
    \end{center}
  \end{figure}
\end{center}

\begin{center}
  \begin{figure}[p]
    \begin{center}
      \includegraphics[scale=0.8]{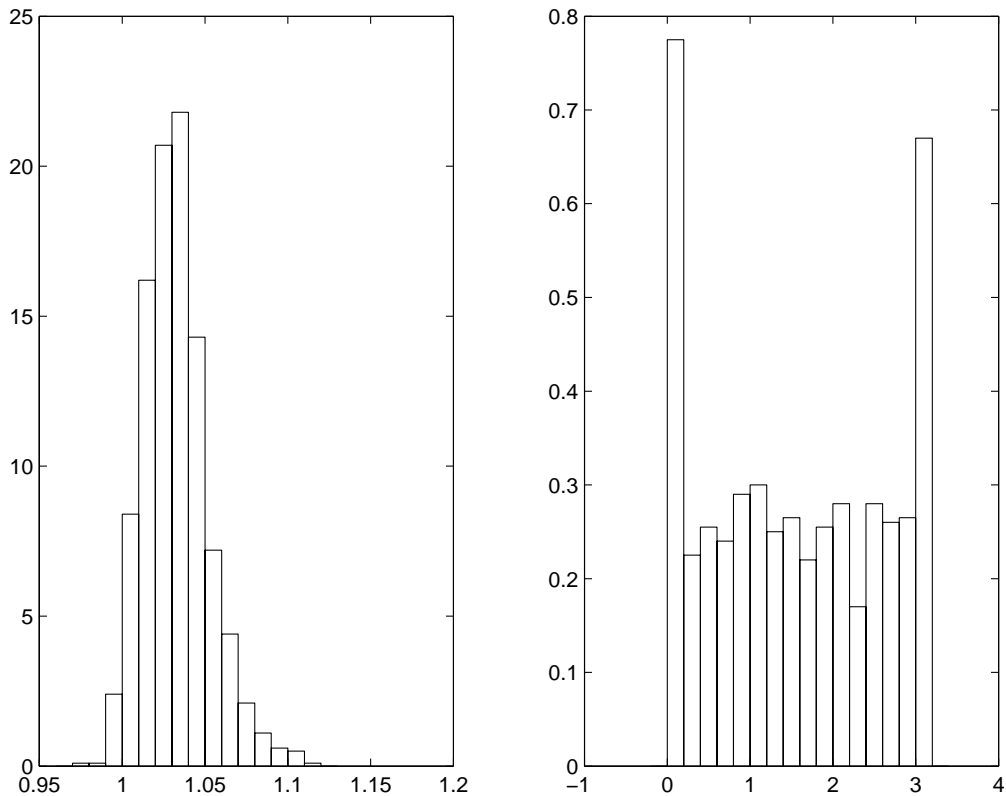}
      \caption{Here $1000$ i.i.d.\ realizations of $\sqrt{n}\,\bM$ where
        simulated where $\bM\sim\cU(\cM_n)$ with $n=300$. The first plot is
        the histogram of $\ABS{\lambda_2(\sqrt{n}\,\bM)}$, i.e.\ the module of
        the sub--dominant eigenvalue $\lambda_2(\sqrt{n}\,\bM)$. The second
        plot is the histogram of
        $\ABS{\mathrm{Phase}(\lambda_2(\sqrt{n}\,\bM))}$. Recall that the
        spectrum is symmetric with respect to the real axis since the matrices
        are real. }
      \label{fi:subdom1}  
    \end{center}
  \end{figure}
\end{center}

\begin{center}
  \begin{figure}[p]
    \begin{center}
      \includegraphics[scale=0.8]{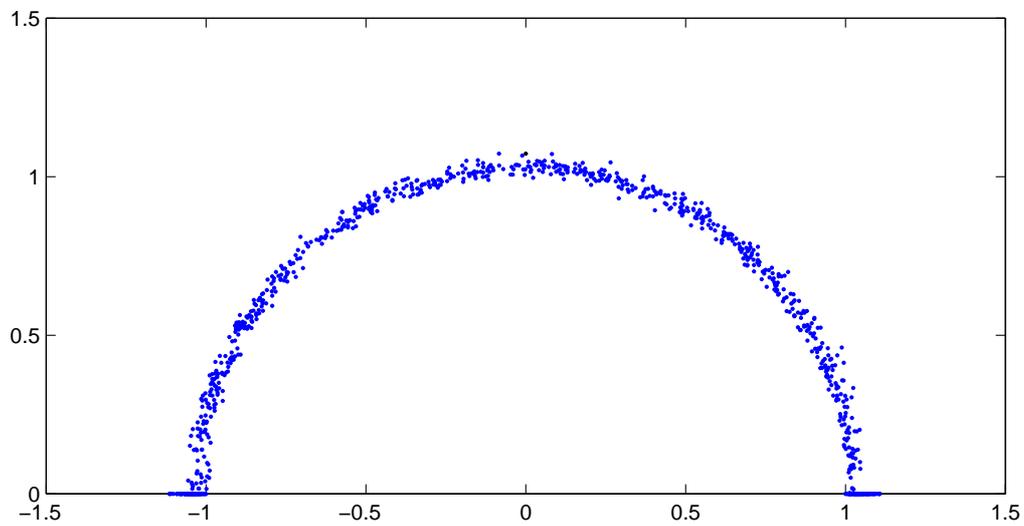}
      \caption{Here we reused the sample used for figure \ref{fi:subdom1}. The
        graphic is a plot of the $1000$ i.i.d.\ realizations of the
        sub--dominant eigenvalue $\lambda_2(\sqrt{n}\,\bM)$. Since we deal with
        real matrices, the spectrum is symmetric with respect to the real
        axis, and we plotted
        $(\mathrm{RealPart}(\lambda_2),|\mathrm{ImaginaryPart}(\lambda_2)|)$ in
        the complex plane.}
      \label{fi:subdom2}  
    \end{center}
  \end{figure}
\end{center}

\end{document}